\documentclass[a4,12pt]{amsart}
\usepackage{amssymb,amstext,amsmath,amscd,amsthm,amsfonts,enumerate,latexsym}
\usepackage{color}
\usepackage[dvipdfmx]{graphicx}
\usepackage[all]{xy}
\theoremstyle{plain}
\newtheorem{thm}{Theorem}[section]

\newtheorem*{thm*}{Theorem}
\newtheorem*{cor*}{Corollary}

\newtheorem{prop}[thm]{Proposition}

\newtheorem{lem}[thm]{Lemma}
\newtheorem{cor}[thm]{Corollary}

\newtheorem{claim}{Claim}
\newtheorem*{claim*}{Claim}

\theoremstyle{definition}

\newtheorem{ex}[thm]{Example}

\newtheorem{rem}[thm]{Remark}

\newtheorem{fact}[thm]{Fact}

\theoremstyle{remark}

\newtheorem*{proof of claim}{{\sl Proof of Claim}}

\numberwithin{equation}{thm}

\def\Tor{\operatorname{Tor}}

\def\mod{\mathrm{mod}}

\newcommand{\rme}{\mathrm{e}}

\newcommand{\calC}{\mathcal{C}}

\newcommand{\calG}{\mathcal{G}}

\newcommand{\calL}{\mathcal{L}}

\newcommand{\calR}{\mathcal{R}}
\newcommand{\calS}{\mathcal{S}}

\newcommand{\fka}{\mathfrak{a}}

\newcommand{\fkm}{\mathfrak{m}}

\newcommand{\fkq}{\mathfrak{q}}

\newcommand{\mapright}[1]{%
\smash{\mathop{%
\hbox to 1cm{\rightarrowfill}}\limits^{#1}}}

\newcommand{\mapleft}[1]{%
\smash{\mathop{%
\hbox to 1cm{\leftarrowfill}}\limits_{#1}}}

\def\depth{\operatorname{depth}}

\def\Ass{\operatorname{Ass}}
\def\height{\mathrm{ht}}

\def\sgn{\operatorname{sgn}}

\title[integrally closed ideals of reduction number three]{Integrally closed ideals of reduction number three}
\author{Shinya Kumashiro}
\address{National Institute of Technology (KOSEN), Oyama College, 771 Nakakuki, Oyama, Tochigi, 323-0806, Japan}
\email{skumashiro@oyama-ct.ac.jp}

\thanks{The author was supported by JSPS KAKENHI Grant Number JP21K13766.}

\begin{document}

\begin{abstract}
In a Cohen-Macaulay local ring $(A, \mathfrak{m})$, we study the Hilbert function of an integrally closed $\mathfrak{m}$-primary ideal $I$ whose reduction number is three. With a mild assumption we give an inequality $\ell_A(A/I) \ge \mathrm{e}_0(I) - \mathrm{e}_1(I) + \dfrac{\mathrm{e}_2(I) + \ell_A(I^2/QI)}{2}$, where $\mathrm{e}_i(I)$ denotes the $i$th Hilbert coefficients and $Q$ denotes a minimal reduction of $I$. The inequality is located between inequalities of Itoh and Elias-Valla. Furthermore our inequality becomes an equality if and only if the depth of the associated graded ring of $I$ is larger than or equal to $\dim A-1$. We also study the Cohen-Macaulayness of the associated graded rings of determinantal rings.
\end{abstract}

\maketitle

\section{Introduction}\label{section1}

The purpose of this paper is to study the Hilbert function of ideals. Let $(A, \fkm)$ be a Cohen-Macauay local ring of dimension $d$, and let $I$ be an $\fkm$-primary ideal. Then the {\it Hilbert function} of $I$ is defined by the length $\ell_A(A/I^{n+1})$ of the residue ring $A/I^{n+1}$ for $n\ge 0$. It is well-known that there exist integers $\rme_0(I), \rme_1(I), \dots, \rme_d(I)$ such that the Hilbert function agrees with a polynomial 
\[
\ell_A (A/I^{n+1})=\rme_0(I)\binom{n+d}{d}-\rme_1(I)\binom{n+d-1}{d-1}+ \cdots + (-1)^{d}\rme_d(I)
\]
in $n$ of degree $d$ for $n\gg 0$. $\rme_0(I), \rme_1(I), \dots, \rme_d(I)$ are called the {\it Hilbert coefficients} of $I$. 
Because the difference $\ell_A(A/I^{n+1})-\ell_A(A/I^{n})=\ell_A(I^n/I^{n+1})$ is the length of the $n$-th homogeneous component of the associated graded rings, it is natural to expect that the Hilbert function reflects the structures of the Rees algebras and the associated graded rings. 
One of the famous results in this direction is the following due to Northcott, Huneke, and Ooishi (\cite{No, H, O}). In what follows, throughout this paper, let $(A, \fkm)$ be a Cohen-Macaulay local ring of dimension $d \ge 2$ and, let $I$ be an $\fkm$-primary ideal. 
For simplicity, assume that $A/\fkm$ is infinite.

\begin{fact}\label{Northcott} {\rm (\cite{No, H, O})}
With the assumptions above, we obtain $\ell_A(A/I)\ge \rme_0(I) - \rme_1(I)$. The equality holds if and only if $I^2=QI$ for some (any) minimal reduction $Q\subseteq I$. When this is the case, the Rees algebra of $I$ and the associated graded ring of $I$ are Cohen-Macaulay rings.
\end{fact}

With the above result as a prototype, Sally \cite{S} explored the equality $\ell_A(A/I)= \rme_0(I) - \rme_1(I)+1$, and Goto, Nishida, and Ozeki \cite{GNO, GNO2} finally characterized this equality by using the notion of the Sally modules. In particular, they proved that $I^3=QI^2$ and $\fkm I^2\subseteq QI$ if the equality $\ell_A(A/I)= \rme_0(I) - \rme_1(I)+1$ holds.

Moreover, when $I$ is integrally closed, the following is known.

\begin{fact} {\rm (\cite[Theorem 12]{I}, \cite[Theorem 2.1]{EV}, \cite[Corollary 2.10]{OR})} \label{ORossi} 
Suppose that $I$ is an integrally closed ideal. Then the following two inequalities hold:
\begin{enumerate}[{\rm (a)}] 
\item $\ell_A(A/I)\le \rme_0(I) - \rme_1(I) + \rme_2(I)$;
\item $\ell_A(A/I)\ge \rme_0(I) - \rme_1(I) + \ell_A(I^2/QI)$.
\end{enumerate} 
The inequality (b) becomes an equality if and only if $I^3=QI^2$ for some (any) minimal reduction $Q\subseteq I$. When this is the case, $\rme_2(I)=\ell_A(I^2/QI)$ and the associated graded ring of $I$ is a Cohen-Macaulay ring.
\end{fact}

In addition, Ozeki and Rossi \cite[Theorem 1.2]{OR} determined the structure of the Sally modules of integrally closed ideals if the equality $\ell_A(A/I)= \rme_0(I) - \rme_1(I) + \ell_A(I^2/QI)+1$ holds. They also showed that, when this is the case, $I^4=QI^3$ and $\fkm I^3 \subseteq QI^2$ hold.

Based on these above results, in this paper, we investigate the Hilbert function of an integrally closed ideal $I$ such that $I^4=QI^3$ for some minimal reduction $Q\subseteq I$. Note that the {\it reduction number}, the minimal number $n\ge 0$ satisfying $I^{n+1}=QI^n$, depends on the choice of a minimal reduction $Q$ in general (see for example \cite{Hu, Ma1, Ma2}).  
The main result of this paper is stated as follows. 

\begin{thm} \label{mainmain}
Suppose that $I$ is an integrally closed ideal such that $I^4=QI^3$ and $\fkm I^3\subseteq QI^2$ for some minimal reduction $Q\subseteq I$. Then 
\begin{align}\label{maineq} 
\ell_A(A/I) \ge \rme_0(I) - \rme_1(I) + \dfrac{\rme_2(I) + \ell_A(I^2/QI)}{2}
\end{align}
holds. The equality holds if and only if $\depth \calG(I) \ge d-1$, where $\calG(I)$ denotes the associated graded ring of $I$. Furthermore, if the inequality {\rm (\ref{maineq})} becomes an equality, then we have 
\begin{align*} 
\ell_A(A/I^{n+1})=&\rme_0(I)\binom{n+d}{d}-[\rme_0(I)-\ell_A (A/I)+\ell_A(I^2/QI)+\ell_A(I^3/QI^2)]\binom{n+d-1}{d-1}\\
&+[\ell_A(I^2/QI) + 2\ell_A(I^3/QI^2)]\binom{n+d-2}{d-2}- \ell_A(I^3/QI^2)\binom{n+d-3}{d-3}
\end{align*}
for all $n\ge 1$. (If $d=2$, then we regard as $\binom{n+d-3}{d-3}=0$.)
\end{thm}

Note that for the case where $I=\fkm$, the inclusion $\fkm I^3\subseteq QI^2$ is automatically satisfied if $I^4=QI^3$. Hence, by noting that $\ell_A(I^2/QI)=\rme_0(I) + (d-1)\ell_A(A/I) - \ell_A(I/I^2)$ (see Remark \ref{remrem}), Theorem \ref{mainmain} can be rephrased as follows.

\begin{cor}\label{corcorcor} 
Suppose that $\fkm^4=Q\fkm^3$ for some minimal reduction $Q$. Then 
 \[
3\rme_0(\fkm)-2\rme_1(\fkm)+\rme_2(\fkm)+d-3 \le v(A)
 \]
 holds, where $v(A)$ denotes the embedding dimension $\ell_A(\fkm/\fkm^2)$. The equality holds if and only if $\depth \calG(\fkm)\ge d-1$.
\end{cor}

In Section \ref{Main results} we prove Theorem \ref{mainmain} and Corollary \ref{corcorcor}. In Section \ref{det} we study examples arising from determinantal rings.



\section{Proof of Theorem  \ref{mainmain}}\label{Main results}

Throughout this section, let $(A, \fkm)$ be a Cohen-Macaulay local ring of dimension $d\ge 2$, and let 
 $I$ be an $\fkm$-primary ideal of $A$. Suppose that $A/\fkm$ is infinite. Choose a parameter ideal $Q\subseteq I$ such that $I^{n+1}=QI^n$ for some $n \ge 0$. 
Set
\begin{align*} 
\calR(I)=A[It] \subseteq A[t] \quad \text{and} \quad  \calG(I)=\calR(I)/I\calR(I)\cong \bigoplus_{n\ge 0}I^n/I^{n+1},
\end{align*}
where $A[t]$ is the polynomial ring over $A$.  $\calR(I)$ and $\calG(I)$ are called the {\it Rees algebra} of $I$ and the {\it associated graded ring} of $I$, respectively.
To investigate the Hilbert function in connection with the structures of $\calR(I)$ and $\calG(I)$, we further define the following finitely generated graded $\calR(Q)$-modules:
\begin{align*} 
&S=\calS_Q (I)=I \calR(I)/I \calR(Q)\cong \bigoplus_{n\ge 1} I^{n+1}/Q^{n}I; \\
&L=\calL_Q (I)= \calR(Q)S_1\cong \bigoplus_{n\ge 1} Q^{n-1}I^{2}/Q^nI;\\
&C=\calC_Q (I)=(I^2 \calR(I)/I^2 \calR(Q))(-1)\cong \bigoplus_{n\ge 2} I^{n+1}/Q^{n-1}I^2.
\end{align*}
$S$ is called the {\it Sally module} of $I$ with respect to $Q$ (\cite{V}). $L$ and $C$ are a part of the filtration of $S$ defined by Vaz Pinto (\cite{VP}). Hence we have the graded exact sequence
\[
0 \to L \to S \to C \to 0.
\]
The importance of the notion of Sally modules is to describe a correction term of the Hilbert function of $I$ (see for example \cite[Proposition 2.2(2)]{GNO}). Furthermore, the filtration of Vaz Pinto is effective as follows if $Q\cap I^2=QI$ holds.

\begin{prop}
\label{basic}
Suppose that $Q\cap I^2=QI$ holds. Then we have the following.
\begin{enumerate}[{\rm (a)}] 
\item {\rm (\cite[Proposition 2.8]{OR})} We have  
\begin{align*} 
\ell_A(A/I^{n+1})=&\rme_0(I)\binom{n+d}{d}-[\rme_0(I)-\ell_A (A/I)+\ell_A(I^2/QI)]\binom{n+d-1}{d-1}\\
&+\ell_A(I^2/QI)\binom{n+d-2}{d-2}- \ell_A(C_n) 
\end{align*}
for all $n\ge 0$.
\item {\rm (\cite[Proposition 2.2]{OR})} $\Ass_{\calR(Q)}C\subseteq \{\fkm \calR(Q)\}$. In particular, either $C=0$ or $\dim C=d$ holds.
\item {\rm (\cite[Lemma 2.11]{OR})} $\depth \calG(I)\ge d-1$ if and only if either $C=0$ holds or $C$ is a ($d$-dimensional) Cohen-Macaulay $\calR(Q)$-module. If $C$ is neither $0$ nor a Cohen-Macaulay $\calR(Q)$-module, then $\depth \calG(I)= \depth_{\calR(Q)} C-1$.
\end{enumerate} 
\end{prop}

\begin{rem} \label{rem1} {\rm (\cite[Theorem 1]{I2})}
If $I$ is an integrally closed ideal, then $Q\cap I^2=QI$ holds.
\end{rem}

Now let us prove Theorem \ref{mainmain}.

\begin{proof}[Proof of Theorem \ref{mainmain}]
If $C=0$, then $\rme_1(I)=\rme_0(I)-\ell_A(A/I) + \ell_A(I^2/QI)$ and $\rme_2(I)=\ell_A(I^2/QI)$ by Proposition \ref{basic}(a). It follows that (\ref{maineq}) becomes an equality if $C=0$. When this is the case, $\depth \calG(I)\ge d-1$ by Proposition \ref{basic}(c). Thus we may assume $C\ne 0$. Then our assumption can be rephrased as follows.

\begin{claim}\label{claim1}{\rm (\cite[Lemma 2.1 (4) and (5)]{OR})}
$I^4=QI^3$ and $\fkm I^3\subseteq QI^2$ if and only if $C=\calR(Q)C_2$ and $\fkm C=0$.
\end{claim}

Hence $C$ is a finitely generated $\calR(Q)/\fkm \calR(Q)$-module generated in degree $2$. It follows that $\Ass_{\calR(Q)/\fkm \calR(Q)} C=\{0\}$ by Proposition \ref{basic}(b). Set $P=\calR(Q)/\fkm \calR(Q)$. Note that $P$ is isomorphic to the polynomial ring $(A/\fkm)[X_1, \dots, X_d]$, because $A$ is a Cohen-Macaulay local ring and $Q$ is a parameter ideal. Therefore, by the graded version of Bourbaki's theorem (see for examples \cite[Corollary 2.4]{K} or \cite[Theorem 2.1]{HKS}), there exist a graded ideal $\fka$ of $P$ and an integer $m\in \mathbb{Z}$ such that 
\begin{align}\label{Bourbaki}
0 \to P(-2)^{r-1} \to C \to \fka(m) \to 0
\end{align}
is a graded exact sequence, where $r$ is the rank of $C$ as a $P$-module. In other words, we have a graded exact sequence 
\[
0 \to P(-2)^{r-1} \to C \to P(m) \to (P/\fka)(m) \to 0.
\]
We can choose $\fka$ such that $\height_P \fka\ge 2$ because $P$ is a factorial domain. Then $\ell_A((P/\fka)_n)$ is a polynomial function of degree less than $d-2$ for $n\gg0$. Hence we obtain that 
\begin{align*} 
\ell_A(C_n)=&(r-1)\ell_A(P_{n-2}) + \ell_A(P_{n+m}) - \ell_A((P/\fka)_{n+m})\\
=& (r-1)\binom{n-2+d-1}{d-1} + \binom{n+m+d-1}{d-1}  - \ell_A((P/\fka)_{n+m})\\
=& r\binom{n+d-1}{d-1} +[-2(r-1)+m]\binom{n+d-2}{d-2}- (\text{a polynomial of degree $<d-2$}) 
\end{align*}
for all $n\gg0$. By Proposition \ref{basic}(a), it follows that 
\begin{small} 
\begin{align*} 
\ell_A(A/I^{n+1})=& \rme_0(I)\binom{n+d}{d}-[\rme_0(I)-\ell_A (A/I)+\ell_A(I^2/QI)+r]\binom{n+d-1}{d-1}\\
&+[\ell_A(I^2/QI)+2(r-1)-m]\binom{n+d-2}{d-2}\\
&+ (\text{a polynomial of degree $<d-2$})\\
\end{align*}
\end{small}
for all $n\gg 0$. 
Therefore, we obtain 
\begin{align*} 
\rme_1(I)&= \rme_0(I)-\ell_A (A/I)+\ell_A(I^2/QI)+r \quad \text{and} \\
\rme_2(I)&=\ell_A(I^2/QI)+2(r-1)-m.
\end{align*}
It follows that 
\begin{align} \label{proofeq}
\rme_2(I)=\ell_A(I^2/QI)+2[\rme_1(I) - \rme_0(I) + \ell_A (A/I) - \ell_A(I^2/QI)]-(m+2).
\end{align}

On the other hand, Claim \ref{claim1} and (\ref{Bourbaki}) imply that $\fka$ is generated in degree $m+2$. Hence we have $m+2\ge 0$ because $\fka$ is an ideal of the polynomial ring $P$. Hence we obtain the inequality (\ref{maineq}) by  (\ref{proofeq}). 

Assume that the inequality (\ref{maineq}) becomes an equality. Then $m+2=0$, i.e. $\fka$ is generated in degree $0$, whence $\fka=P$. It follows that (\ref{Bourbaki}) splits and $C\cong P(-2)^r$. Thus we have $\depth \calG(I)\ge d-1$ by Proposition \ref{basic}(c). Conversely, assume that $\depth \calG(I)\ge d-1$. Then $C$ is a Cohen-Macaulay $\calR(Q)$-module, thus $C$ is a Cohen-Macaulay module as a $P$-module. Because $P$ is a polynomial ring over the field $A/\fkm$, $C$ is a graded free $P$-module of rank $r$. By Claim \ref{claim1}, it follows that $C\cong P(-2)^r$. Therefore, the inequality (\ref{maineq}) becomes an equality by substituting 
\[
\ell_A(C_n) = r \binom{n-2+d-1}{d-1}=r\binom{n+d-1}{d-1}-2r\binom{n+d-2}{d-2}+r\binom{n+d-3}{d-3}
\]
for all $n\ge 1$ for the equation in Proposition \ref{basic}(a).

When this is the case, because $C\cong P(-2)^r$, we obtain 
\begin{align*} 
r=&\ell_A(C_2/\fkm C_2)=\ell_A(I^3/[QI^2~+~\fkm I^3])=\ell_A(I^3/QI^2) \quad \text{and}\\
\ell_A(C_n)=&r \binom{n-2+d-1}{d-1} \quad \text{for all $n\ge 1$.}
\end{align*} 
It allows us to calculate the Hilbert function by using Proposition \ref{basic}(a).
\end{proof}

Corollary \ref{corcorcor} follows from Theorem \ref{mainmain} and the following remark.

\begin{rem} \label{remrem}
Let $(A, \fkm)$ be a Cohen-Macaulay local ring, and let $I$ be an $\fkm$-primary ideal. Suppose that $Q\subseteq I$ is a parameter ideal such that $Q$ is a reduction of $I$. Then $\ell_A(I^2/QI)=\rme_0(I) + (d-1)\ell_A(A/I) - \ell_A(I/I^2)$ holds.
\end{rem}

\begin{proof} 
Because we have the following inclusions
\begin{small} 
\[
\xymatrix{
           &    A      & \\
& I \ar@{-}[u]   &  \\
I^2 \ar@{-}[ur] &  & Q \ar@{-}[ul]\\
& QI, \ar@{-}[ur] \ar@{-}[ul] 
}
\]
\end{small}
we obtain 
\begin{align*} 
\begin{split} \label{note}
\ell_A(I^2/QI)=&\ell_A(A/Q) + \ell_A(Q/QI) - \ell_A (A/I) - \ell_A(I/I^2)\\
=&\rme_0(I) + (d-1)\ell_A(A/I) - \ell_A(I/I^2),
\end{split}
\end{align*}
where the second equality follows from the isomorphism 
\[
Q/QI\cong Q/Q^2\otimes_A A/I\cong (A/I)^d. 
\]
\end{proof}

We can also characterize when the difference of the inequality (\ref{maineq}) is $1/2$.

\begin{cor}\label{2.42.4}
Suppose that $I$ is an integrally closed $\fkm$-primary ideal such that $I^4=QI^3$ and $\fkm I^3 \subseteq QI^2$. Then the following are equivalent:
\begin{enumerate}[{\rm (a)}] 
\item $\ell_A(A/I) = \rme_0(I) - \rme_1(I) + \dfrac{\rme_2(I) + \ell_A(I^2/QI)+1}{2}$;
\item there exists a graded exact sequence 
\[
0 \to P(-2)^{r-1} \to C \to (X_1, X_2, \dots, X_c) (-1)\to 0,
\]
where $r$ denotes the rank of $C$ as a $P$-module, $X_1, X_2, \dots, X_d$ are variants of $P \cong (A/\fkm)[X_1, X_2, \dots, X_d]$, and $c=\ell_A(I^3/QI^2)-r+1\ge 2$.
\end{enumerate} 
When this is the case, $\depth \calG(I)=d-c$ and 
\begin{align*} 
\ell_A(A/I^{n+1})=&\rme_0(I)\binom{n+d}{d}-[\rme_0(I)-\ell_A (A/I)+\ell_A(I^2/QI)+r]\binom{n+d-1}{d-1}\\
&+[\ell_A(I^2/QI) + 2r-1]\binom{n+d-2}{d-2}- (r-1)\binom{n+d-3}{d-3} \\
& + \binom{n+d-c-1}{d-c-1} - \binom{n+d-c-2}{d-c-2}
\end{align*}
for all $n\ge 1$. (we regard as $\binom{n+i}{i}=0$ if $i<0$.)
\end{cor}

\begin{proof}
 (a) $\Rightarrow$ (b): If the assertion of (a) holds, then $m+2=1$ by (\ref{proofeq}). It follows that $\fka$ is generated in degree $1$. Hence we obtain $\fka\cong (X_1, X_2, \dots, X_c)$ for some $c>0$. Then $c=\ell_A(I^3/QI^2)-r+1$ holds by (\ref{Bourbaki}), and $c\ge 2$ holds by Theorem \ref{mainmain}.

 (b) $\Rightarrow$ (a): Because 
 \[
 \ell_A(C_n)=(r-1)\ell_A(P_{n-2}) + \ell_A(P_{n-1}) - \ell_A([P/(X_1, \dots, X_c)]_{n-1})
 \] 
 for all $n \ge 1$, we can calculate $\ell_A(A/I^{n+1})$ in the same way as in the proof of Theorem \ref{mainmain}.

 When this is the case, $\depth_{\calR(Q)} C=\depth (X_1, \dots, X_c)=d-c+1<d$ because $C$ is not a Cohen-Macaulay $\calR(Q)$-module by Theorem \ref{mainmain}. Hence $\depth \calG(I)=d-c$ by Proposition \ref{basic}(c). The Hilbert function can be calculated by Proposition \ref{basic}(a).
\end{proof}

In the rest of this paper, we note some examples of Theorem \ref{mainmain}.

\begin{ex}{\rm (\cite[Example 3.2]{CPR}, \cite[Theorem 3.12]{HH}, \cite[Example 3.8]{MORT})}
Let $A=K[[X, Y, Z]]$ be the formal power series ring over a field $K$. Let $\fkm$ be the maximal ideal of $A$ and $N=(X^4, X(Y^3+Z^3), Y(Y^3+Z^3), Z(Y^3+Z^3))$. Set $I=N+\fkm^5$. Suppose that $K$ is a field of characteristic $\ne 3$. Then the following are true:
\begin{enumerate}[{\rm (a)}] 
\item $I$ is a normal $\fkm$-primary ideal, i.e. $I^n$ is an integrally closed ideal for all $n> 0$;
\item $I^4=QI^3$ and $\ell_A(I^3/QI^2)=1$ for any minimal reduction $Q$ of $I$;
\item The Hilbert series of $I$ is 
\[
\frac{31+43t+t^2+t^3}{(1-t)^3},
\]
whence $\ell_A(A/I)=31$, $\ell_A(I/I^2)=136$, $\rme_0(I)=76$, $\rme_1(I)=48$, $\rme_2(I)=4$, and $\rme_3(I)=1$. Thus the inequality (\ref{maineq}) becomes an equality;
\item $\depth \calG (I)=\dim A-1=2$. 
\end{enumerate} 

\end{ex}

\begin{ex}{\rm (\cite[Theorem 5.2]{OR})}
Let $m>0$, $d\ge 2$, and $D=K[[\{X_j\}_{1\le j \le m},  \{Y_i\}_{1\le i\le d}, \{Z_i\}_{1\le i\le d}]]$ be the formal power series  ring over a field $K$. Let
{\small
\[
\fka=(X_1, \dots, X_{m}){\cdot}(X_1, \dots, X_{m}, Y_1, \dots, Y_d) + (Y_iY_j \mid 1\le i, j\le d, i\ne j) + (Y_i^3-Z_iX_{m} \mid 1\le i \le d).
\]
}
Set $A=D/\fka$.  We denote by $x_*, y_*, z_*$ the image of $X_*, Y_*, Z_*$ into $A$. Let $\fkm$ be the maximal ideal of $A$ and 
$Q=(z_1, \dots, z_d)$. Then the following assertions are true:
\begin{enumerate}[{\rm (a)}] 
\item $A$ is a Cohen-Macaulay local ring of dimension $d$;
\item $\fkm^4=Q\fkm^3$;
\item $\rme_0(\fkm)=m+2d+1, \rme_1(\fkm)=m+3d+1, \rme(\fkm)=d+1, \text{ and } \rme_i(\fkm)=0 \quad \text{for all $3\le i \le d$}$. Hence 
$3\rme_0(\fkm)-2\rme_1(\fkm)+\rme_2(\fkm)+d-3 = v(A)-1$, which is the same as the equality in Corollary \ref{2.42.4} (a), holds;
\item $\depth \calG(\fkm) = 0$.
\end{enumerate}  
\end{ex}


\section{The associated graded rings of  determinantal rings}\label{det}

In this section, we explore examples of Theorem \ref{mainmain} arising from determinantal rings. Let $R$ be a commutative ring and $\{X_{ij}\}_{1\le i\le s, 1\le j\le t}$ be elements of $R$. Assume $2\le s\le t$, and $I_s(X_{ij})$ denotes the ideal of $R$ generated by $s\times s$-minors of the $s\times t$ matrix $(X_{ij})$. Set $A=R/I_s(X_{ij})$. Let $L$ be the ideal of $R$ generated by $\{X_{ij}\}_{1\le i\le s, 1\le j\le t}$ and $Q$ be the ideal of $R$ generated by
\begin{align*} 
&\text{$X_{ij}$ \quad ($j-i<0$ or $j-i>t-s$) \quad and}\\
&\text{$X_{ij}-X_{i-1, j-1}$ \quad ($2\le i \le s$ and $0\le j-i \le t-s$).}
\end{align*}
Set $I=LA$ and $\fkq=QA$. We denote by $x_{ij}$ the image of $X_{ij}\in R$ into $A$.
With these assumptions and notation we study the Cohen-Macaulayness of $\calG(I)$, see Theorem \ref{thm2}. Let us prepare a lemma.

\begin{lem}\label{32} 
For all $j_1, j_2, \dots, j_s\in \{1,2, \dots, s\}$ and $k_1, k_2, \dots, k_s\in \{1,2, \dots, t\}$, one of the following holds.
\begin{enumerate}[{\rm (a)}] 
\item $x_{j_1, k_1}x_{j_2, k_2} \cdots x_{j_s, k_s}\in \fkq I^{s-1}$. 
\item $x_{j_1, k_1}x_{j_2, k_2} \cdots x_{j_s, k_s} \equiv x_{1, i_1}x_{2, i_2} \cdots x_{s, i_s} \ \mod \ \fkq I^{s-1}$ for some $1\le i_1<i_2<\cdots <i_s\le t$. 
\end{enumerate} 
\end{lem}

\begin{proof}
By definition of $\fkq$, we may assume that $0\le k_u-j_u\le t-s$ for all $1\le u \le s$. If $j_u\ge 2$, then 
\[
x_{j_1, k_1} \cdots x_{j_u, k_u}\cdots x_{j_s, k_s} \equiv x_{j_1, k_1} \cdots x_{j_{u}-1, k_{u}-1}\cdots x_{j_s, k_s} \ \mod \ \fkq I^{s-1}
\]
 because $x_{j_u, k_u}-x_{j_{u}-1, k_{u}-1}\in \fkq$.  Hence we may assume that $j_1=j_2=\cdots=j_s=1$. By reordering $x_{1, k_1}, x_{1, k_2}, \dots, x_{1, k_s}$ if necessary, we may also assume that $1 \le k_1 \le k_2 \le \cdots \le k_s \le t-s+1$. Then we obtain 
 \[
 x_{1, k_1}x_{1, k_2} \cdots x_{1, k_s} \equiv  x_{1, k_1}x_{2, k_2+1} \cdots x_{s, k_s+s-1} \ \mod \ \fkq I^{s-1}
 \]
in a  similar way to the discussion above. This shows the assertion because $1\le k_1<k_2+1<\cdots <k_s+s-1\le t$. 
\end{proof}

\begin{thm}\label{thm2}
The following equalitiies hold true.
\begin{enumerate}[{\rm (a)}]  
\item $I^s=\fkq I^{s-1}$.
\item {\rm (cf. \cite[(9.14) Theorem]{BV})} Suppose that $R$ is a Noetherian local ring. If $\{X_{ij}\}_{1\le i\le s, 1\le j \le t}$ is a regular sequence of $R$, then $\fkq \cap I^{n+1}=\fkq I^n$ holds for all $n\ge 0$. Hence, $\calG(I)$ is a Cohen-Macaulay ring.
\end{enumerate}
\end{thm}

\begin{proof}
(a): Set $S=\{ (i_1, i_2, \dots,i_s) \mid 1\le i_1<i_2<\cdots <i_s\le t\}$. We regard $S$ as a totally ordered set in lexicographic order. Set $S=\{\mathbf{a}_1<\mathbf{a}_2<\cdots <\mathbf{a}_{\binom{t}{s}}\}$. For $\mathbf{a}=(i_1, i_2, \dots,i_s) \in S$, let $x_{\mathbf{a}}$ denote $x_{1, i_1}x_{2, i_2} \cdots x_{s, i_s}$. With the notation, we will prove $x_{\mathbf{a}}\in \fkq I^{s-1}$ for all $\mathbf{a}\in S$ by induction with respect to the lexicographic order (this shows $I^s\subseteq \fkq I^{s-1}$ by Lemma \ref{32}). 

In what follows, all $\equiv$ are considered modulo $\fkq I^{s-1}$. For $\mathbf{a}=(i_1, i_2, \dots,i_s) \in S$, $\det \mathbf{a}$ denotes the  determinant of the submatrix of $(x_{ij})$ with respect to the columns indexed by $i_1, i_2, \dots,i_s$ and the whole rows.

$x_{\mathbf{a}_1}\in \fkq I^{s-1}$ follows from $0= \det(1,2,\dots, s) \equiv x_{11}x_{22}\cdots x_{ss}$ because $x_{ij}\in \fkq$ if $j-i<0$. 
Assume $1<u$ and $x_{\mathbf{a}_1}, \ldots, x_{\mathbf{a}_{u-1}}\in \fkq I^{s-1}$. Set $\mathbf{a}_u=(i_1, i_2, \dots,i_s)$. Then 
\[
0 = \det(i_1, i_2, \dots,i_s)=\sum_{\sigma\in S_s} \sgn(\sigma) x_{\sigma(1), i_1} x_{\sigma(2), i_2} \cdots x_{\sigma(s), i_s}.
\]
Assume $\sigma(1)>1$. Then, by the same way as in the proof of Lemma \ref{32}, we have $x_{\sigma(1), i_1} x_{\sigma(2), i_2} \cdots x_{\sigma(s), i_s} \equiv x_{\mathbf{a}'}$ for some $\mathbf{a}'=(i'_1, \dots, i'_s)$ such that $i'_1<i_1$. Thus $\mathbf{a}'<\mathbf{a_u}$. Similarly, assume that for $1<v< s$, $\sigma(1)=1, \ldots, \sigma(v-1)=v-1$, and $\sigma(v)>v$. Then we have $x_{\sigma(1), i_1} x_{\sigma(2), i_2} \cdots x_{\sigma(s), i_s} \equiv x_{\mathbf{a}'}$ for some $\mathbf{a}'<\mathbf{a_u}$. 
It follows that 
\[
0 = \det(i_1, i_2, \dots,i_s)\equiv x_{1, i_1} x_{2, i_2} \cdots x_{s, i_s}=x_{\mathbf{a}_u}
\]
by induction hypothesis, whence $x_{\mathbf{a}_u}\in \fkq I^{s-1}$.

(b): We show that $(Q + I_s(X_{ij})) \cap (L^{n+1} + I_s(X_{ij}))\subseteq QL^n + I_s(X_{ij})$ for all $n\ge 0$. If $n+1 \ge s$, then the inclusion follows from $L^{n+1} + I_s(X_{ij}) = QL^n + I_s(X_{ij})$ by (a). Assume that $n+1 <s$. Then
\begin{align*} 
(Q + I_s(X_{ij})) \cap (L^{n+1} + I_s(X_{ij})) &= (Q + I_s(X_{ij})) \cap L^{n+1}  + I_s(X_{ij})\\
&= (Q + (X_{11}, \dots, X_{1s})^s) \cap L^{n+1} + I_s(X_{ij})\\
&= Q \cap L^{n+1} + (X_{11}, \dots, X_{1s})^s + I_s(X_{ij}),
\end{align*}
where the second equality follows from 
{\small
\[
Q+I_s(X_{ij})=Q + I_s\left(
\begin{matrix} 
X_{11} & \cdots & X_{1s} & 0 & \cdots & 0\\
0 & \ddots & & \ddots & \ddots & \vdots \\
\vdots &\ddots &\ddots &&\ddots &0\\
0 & \cdots & 0 & X_{11} & \cdots & X_{1s} 
\end{matrix}
\right).
\]
}

By noting that $L=Q + (X_{11}, \dots, X_{1s})$, we obtain 
\begin{align*} 
\begin{split} 
L^{i+1} &= QL^i + (X_{11}, \dots, X_{1s})^{i+1} \quad \text{and} \quad \\
Q \cap L^{i+1}  &= QL^i + Q(X_{11}, \dots, X_{1s})^{i+1} =QL^i
\end{split}
\end{align*}
for all $i \ge 0$ because Remark \ref{remarkaa}. Hence
\begin{align*} 
Q \cap L^{n+1} + (X_{11}, \dots, X_{1s})^s + I_s(X_{ij}) = QL^n + (X_{11}, \dots, X_{1s})^s + I_s(X_{ij}).
\end{align*} 

On the other hand, by (a), we have $L^s=QL^{s-1} + (X_{11}, \dots, X_{1s})^{s} \subseteq QL^{s-1} + I_s(X_{ij})$. In particular, 
$(X_{11}, \dots, X_{1s})^{s} \subseteq QL^{s-1} + I_s(X_{ij}) \subseteq QL^n + I_s(X_{ij})$.
Thus we obtain 
\[
(Q + I_s(X_{ij})) \cap (L^{n+1} + I_s(X_{ij}))\subseteq  QL^n + (X_{11}, \dots, X_{1s})^s + I_s(X_{ij}) = QL^n + I_s(X_{ij})
\] 
as desired. It follows that $\calG(I)$ is a Cohen-Macaulay ring by \cite[Theorem 1.1]{VV}.
\end{proof}

\begin{rem} \label{remarkaa}
Let $(R, \fkm)$ be a Noetherian local ring. Let $x_1, \dots, x_m, y_1, \dots, y_n\in \fkm$ be a regular sequence of $R$. Set $I=(x_1, \dots, x_m)$ and $J=(y_1, \dots, y_n)$. Then $I^i\cap J^j=I^iJ^j$ for all $i,j\ge 0$.
\end{rem}

\begin{proof}
It is equivalent to saying that $\Tor_1^R (R/I^i, R/J^j)=0$. Because $I^i/I^{i+1}$ (resp. $J^j/J^{j+1}$) is a free $R/I$-module (resp. a free $R/J$-module), by using inductions on $i$ and $j$, it is enough to show that $\Tor_1^R (R/I, R/J)=0$. Set $I_0=R$ and $I_\ell=(x_1, \dots, x_\ell)$ for $1\le \ell \le m$. Then we prove $\Tor_1^R (R/I_\ell, R/J)=0$ for all $0\le \ell \le m$ by induction on $\ell$. The case where $\ell=0$ is clear. Assume $\ell>0$ and $\Tor_1^R (R/I_{\ell-1}, R/J)=0$. Then, by applying the functor $-\otimes_R R/J$ to the exact sequence $0\to R/I_{\ell-1} \xrightarrow{x_\ell} R/I_{\ell-1} \to R/I_\ell \to 0$, we obtain $\Tor_1^R (R/I_{\ell}, R/J)=0$.
\end{proof}

The following is a direct consequence of Theorem \ref{thm2}.

\begin{cor}
Let $s \le t$ be integers. Let $R$ be a regular local ring of dimension $st$ and $(X_{ij} \mid 1\le i \le s, 1 \le j \le t)$ the maximal ideal of $A$. Then $A=R/I_s(X_{ij})$ is a Cohen-Macaulay local ring of dimension $st - (t-s+1)$ (see for example \cite{BV}). Let $\fkm$ be the maximal ideal of $A$. Then we have the following:
\begin{enumerate}[{\rm (a)}]  	
\item there exists a parameter ideal $Q$ of $A$ such that $\fkm^{s}=Q\fkm^{s-1}$;
\item $\calG(\fkm)$ is a Cohen-Macaulay ring.
\end{enumerate}
\end{cor}


\end{document}